\documentclass[10pt]{amsart}
\usepackage{amscd,graphicx,epsfig}
\usepackage[colorlinks=true, pdfstartview=FitV, linkcolor=blue, citecolor=blue, urlcolor=blue]{hyperref}
\usepackage{amssymb}

\title[Automorphisms of the Affine Cremona Group]{On Automorphisms of the Affine Cremona Group}
\author{Hanspeter Kraft and Immanuel Stampfli}

\address{Mathematisches Institut,
Universit\"at Basel, Rheinsprung 21, CH-4051 Basel}
\email{Hanspeter.Kraft@unibas.ch}
\email{Immanuel.Stampfli@unibas.ch}
\date{Final version May 2013}

\thanks{Both authors were partially supported by  Swiss National Science Foundation (Schweizerischer National\-fonds), grant ``Automorphisms of Affine
$n$-Space 137679}

\newtheorem*{mthm}{Main Theorem}

\newtheorem{prop}[subsubsection]{Proposition}
\newtheorem{lem}[subsubsection]{Lemma}
\newtheorem{cor}[subsubsection]{Corollary}

\theoremstyle{definition}

\theoremstyle{remark}
\newtheorem{rem}[subsubsection]{Remark}

\newcommand{\name}[1]{\textsc{#1\/}}
\newcommand{\NN}{{\mathbb N}}
\newcommand{\ZZ}{{\mathbb Z}}
\newcommand{\PP}{{\mathbb P}}

\newcommand{\CC}{{\mathbb C}}

\newcommand{\QQ}{{\mathbb Q}}

\newcommand{\Cst}{{{\mathbb C}^*}}

\renewcommand{\AA}{{\mathbb A}}
\newcommand{\An}{{\mathbb A}^{n}}

\newcommand{\Gn}{{\mathcal G}_{n}}
\newcommand{\TGn}{T{\mathcal G}_{n}}
\newcommand{\Gtwo}{{\mathcal G}_{2}}
\newcommand{\Gthree}{{\mathcal G}_{3}}
\newcommand{\TGtwo}{T{\mathcal G}_{2}}

\newcommand{\Jn}{{\mathcal J}_{n}}

\newcommand{\TTT}{\mathcal T}
\newcommand{\Tn}{{\mathcal T}_{n}}
\newcommand{\simto}{\xrightarrow{\sim}}
\newcommand{\be}{\begin{enumerate}}
\newcommand{\ee}{\end{enumerate}}
\newcommand{\eps}{\varepsilon}

\DeclareMathOperator{\id}{id}

\DeclareMathOperator{\Aff}{Aff}
\DeclareMathOperator{\C}{Cent}
\DeclareMathOperator{\GL}{GL}
\DeclareMathOperator{\SL}{SL}
\DeclareMathOperator{\Aut}{Aut}
\DeclareMathOperator{\Lie}{Lie}
\DeclareMathOperator{\Id}{Id}

\DeclareMathOperator{\Exp}{Exp}

\newcommand{\g}{\mathbf g}
\newcommand{\f}{\mathbf f}
\newcommand{\bd}{\mathbf d}
\newcommand{\bt}{\mathbf t}
\newcommand{\mm}{\mathfrak m}

\newcommand{\OOO}{\mathcal O}

\newcommand{\Cplus}{{\CC^{+}}}

\newcommand{\Xu}{X_{u}(D_{n})}

\renewcommand{\phi}{\varphi}

\begin{document}

\begin{abstract}
We show that every automorphism of the group $\Gn:=\Aut(\An)$ of polynomial automorphisms of 
complex affine $n$-space $\An=\CC^{n}$ 
is inner up to field automorphisms when restricted to the subgroup $\TGn$ of tame automorphisms. 
This generalizes a result of \name{Julie Deserti} who proved this in 
dimension $n=2$ where all automorphisms are tame: $\TGtwo = \Gtwo$. The methods are different, based on
arguments from algebraic groups actions.
\end{abstract}

\maketitle

%%%%%%%%%%%%%%%%%%%%%%%%%%%%%%%%%%%%%%%%%
\subsection{Notation}
Let $\Gn:=\Aut(\An)$ denote the group of polynomial automorphisms of complex affine $n$-space $\An = \CC^{n}$. For an  automorphism $\g$ we use the notation  $\g=(g_{1},g_{2},\ldots,g_{n})$ if 
$$
\g(a) = (g_{1}(a_{1},\ldots,a_{n}),\ldots,g_{n}(a_{1},\ldots,a_{n}))\quad\text{for\ } a=(a_{1},\ldots,a_{n})\in \An
$$ 
where $g_{1},\ldots,g_{n}\in\CC[x_{1},\ldots,x_{n}]$. Moreover, we define the degree of $\g$ by $\deg \g := \max(\deg g_{1},\ldots,\deg g_{n})$. The product of two automorphisms is denoted by $\f \circ \g$.

The automorphisms of the form $(g_{1},\ldots,g_{n})$ where $g_{i}= g_{i}(x_{i},\ldots,x_{n})$ depends only on $x_{i},\ldots,x_{n}$,  form the {\it Jonqui\`ere subgroup} $\Jn \subset \Gn$. Moreover, we have the inclusions $D_{n}\subset \GL_{n}\subset \Aff_{n}\subset \Gn$ where $D_{n}$ is the group of  {\it diagonal automorphisms\/} $(a_{1}x_{1},\ldots,a_{n}x_{n})$ and $\Aff_{n}$ is the group of {\it affine transformations\/} $\g=(g_{1},\ldots,g_{n})$ where all $g_{i}$ have degree 1. The group $\Aff_{n}$ is the semidirect product of $\GL_{n}$ with the commutative unipotent subgroup $\Tn$ of translations. The subgroup $\TGn\subset\Gn$ generated by $\Jn$ and $\Aff_{n}$ is called the group of {\it tame\/} automorphisms.

\begin{mthm}
Let $\theta$ be an automorphism of $\Gn$. Then there is an element $\g\in\Gn$ and a field automorphism $\tau\colon\CC \to\CC$ such that 
$$
\theta (\f) = \tau(\g\circ \f \circ \g^{-1})\text{ for all tame automorphisms } \f\in\TGn.
$$
\end{mthm}
After some preparation in the following sections the proof is given in Section~\ref{proof.subsec}. For $n=2$ where $\TGtwo=\Gtwo$ this result is due to \name{Julie Deserti} \cite{De2006Sur-le-groupe-des-}. In fact, she proved this for any uncountable field $K$ of characteristic zero. Our methods work for any algebraically closed field of characteristic zero.

Recently, \name{Alexei Belov-Kanel} and \name{Jie-Tai Yu} proved that every ind-group automorphism $\Gn \to \Gn$ is inner \cite{BeYu2013On-The-Zariski-Top}.

\par\medskip
%%%%%%%%%%%%%%%%%%%%%%%%%%%%%%%%%%%%%%%%%
\subsection{Ind-group structure and locally finite automorphisms}
The group $\Gn$ has the structure of an ind-group given by $\Gn = \bigcup_{d\geq 1}(\Gn)_{d}$ where $(\Gn)_{d}$ are the automorphisms of degree $\leq d$ (see \cite{Ku2002Kac-Moody-groups-t}).  Each $(\Gn)_{d}$ is an affine variety and $(\Gn)_{d}\subset (\Gn)_{d+1}$ is closed for all $d$. This defines a topology on $\Gn$ where a subset $X \subset \Gn$ is closed (resp. open) if and only if $X \cap (\Gn)_{d}$ is closed (resp. open) in $(\Gn)_{d}$ for all $d$. All subgroups mentioned above are closed subgroups. 

In addition, multiplication $\Gn\times \Gn \to \Gn$ and inverse $\Gn \to \Gn$ are morphisms of ind-varieties where for the latter one has to use the fact that $\deg \f^{-1}\leq (\deg \f)^{n-1}$. This seems to be a classical result for birational maps of $\PP^{n}$ based on B\'ezout's Theorem (see \cite[Corollary~(1.4) and Theorem~(1.5)]{BaCoWr1982The-Jacobian-conje}). It follows that for every subgroup $G \subset \Gn$ the closure $\bar G$ in $\Gn$ is also a subgroup. 

A closed subgroup $G$ contained in some $(\Gn)_{d}$ is called an {\it algebraic subgroup}. In fact, such a $G$ is an affine algebraic group which acts faithfully on $\An$, and for every affine algebraic group $H$ acting on $\An$ the image of $H$ in $\Gn$ is an algebraic subgroup.

A subset $X \subset \Gn$ is called {\it bounded constructible}, if $X$ is a constructible subset of some $(\Gn)_{d}$. 

\begin{lem}\label{alggp.lem}
Let $G\subset \Gn$ be a subgroup and let $X\subset G$ be a subset which is dense in $G$ and bounded constructible. Then $G$ is an algebraic subgroup, and $G = X\circ X$.
\end{lem}

\begin{proof}
By assumption $G \subset \bar X \subset (\Gn)_{d}$ for some $d$ and so $\bar G=\bar X$ is an algebraic subgroup. Moreover, there is a subset $U\subset X$ which is open and dense in $\bar G$. Then $U\circ U = \bar G$, and so $\bar G =  G = X\circ X$.
\end{proof}

An element $\g\in\Gn$ is called {\it locally finite\/} if it induces a locally finite automorphism of the algebra $\CC[x_{1},\ldots,x_{n}]$ of polynomial functions on $\An$. This is equivalent to the condition that the linear span of 
$\{(\g^{m})^{*}(f)\mid m\in\ZZ\}$ is finite dimensional for all $f\in\CC[x_{1},\ldots,x_{n}]$. 

More generally, an action of a group $G$ on an affine variety $X$ is called {\it locally finite\/} if the induced action on the coordinate ring $\OOO(X)$ is locally finite, i.e. for all $f\in\OOO(X)$ the linear span $\langle Gf \rangle$ is finite dimensional. It is easy to see that the image of $G$ in $\Aut(X)$  is dense in an algebraic group $\bar G$ which acts algebraically on $X$. In fact, one first chooses a finite dimensional $G$-stable subspace $W \subset \OOO(X)$ which generates $\OOO(X)$, and then defines $\bar G \subset \GL(W)$ to be the closure of the image of $G$
inside $\GL(W)$.

The next result will be used in the following section. We start again with an action of a group $G$ on an affine variety $X$ and assume that $x_{0}\in X$ is a fixed point. Then we obtain a representation $\tau\colon G \to \GL(T_{x_{0}}X)$ on the tangent space at $x_{0}$, given by $\tau(g):= d_{x_{0}}g$.

\begin{lem}\label{faithful.lem}
Let $G$ act faithfully on an irreducible affine variety $X$.  Assume that $x_{0}\in X$ is a fixed point and that  there is a $G$-stable decomposition $\mm_{x_{0}}=V \oplus \mm_{x_{0}}^{2}$. Then the tangent representation $\tau\colon G \to \GL(T_{x_{0}} X)$ is faithful.
\end{lem}

\begin{proof} 
Let $g\in \ker \tau$. Then $g$ acts trivially on $V$, hence on all powers $V^{j}$ of $V$. This implies that the action of $g$ on $\OOO(X)/\mm_{x_{0}}^{k}$ is trivial for all $k\geq 1$. Since 
$\bigcap_{k}\mm_{x_{0}}^{k}= \{0\}$ the claim follows.
\end{proof}
We remark that a $G$-stable decomposition $\mm_{x_{0}}=V \oplus \mm_{x_{0}}^{2}$ like in the lemma above always exists if $G$ is a reductive algebraic group.

\par\medskip
%%%%%%%%%%%%%%%%%%%%%%%%%%%%%%%%%%%%%%%%%
\subsection{Tori and centralizers}
For the convenience of the reader we recall two important results about fixed point sets of group actions which we will need below. A complex variety $X$ is called {\it $\ZZ/p\ZZ$-acyclic} if  $H_j(X, \ZZ/p\ZZ)= 0$
for $j > 0$ and $H_0(X, \ZZ/p\ZZ)=\ZZ/p\ZZ$. The first result goes back to \name{P.~A. Smith} \cite{Sm1934A-theorem-on-fixed}.

\begin{prop}[Corollary to Theorem 7.5 in \cite{Se2009How-to-use-finite-}] \label{Serre.prop}
Let $G$ be a finite $p$-group and let $X$ be an affine $G$-variety.
If $X$ is $\ZZ/p\ZZ$-acyclic, then so is $X^G$. 
\end{prop}

The second result is due to \name{Fogarty} and describes the {\it tangent cone} $C(X^G, x)$
of the fixed point set $X^G$.

\begin{prop}[Theorem 5.2 in \cite{Fo1973Fixed-point-scheme}]\label{Fogarty.prop}
Let $G$ be a reductive group. If $X$ is an affine $G$-variety, then for each point $x \in X$ we have
$C(X^G, x) =C(X, x)^G$.
\end{prop}

Define $\mu_{k}:=\{\g\in D_{n}\mid \g^{k}=\id\}$. We have 
$\mu_{k}\simeq (\ZZ/k)^{n}$, and $\mu_{\infty}:=\bigcup_{k}\mu_{k}\subset D_{n}$ 
is the subgroup of elements of finite order where $\mu_{\infty}\simeq (\QQ/\ZZ)^{n}$.
The next lemma about the centralizer of $\mu_{k}$ is easy.

\begin{lem}\label{lem1}
For every $k>1$ we have $\C_{\Gn}(\mu_{k}) = \C_{\GL_{n}}(\mu_{k}) = D_{n}$.
\end{lem}

The following result is crucial for the proof of the main theorem.

\begin{prop}\label{crucial.prop}
Let $\mu\subset \Gn$ be a finite subgroup isomorphic to $\mu_{2}$. Then the centralizer $\C_{\Gn}(\mu)$ is a 
diagonalizable algebraic subgroup of $\Gn$, i.e. isomorphic to a closed subgroup of a torus. Moreover $\dim \C_{\Gn}(\mu) \leq n$.
\end{prop}

\begin{proof}
We first remark that $\C_{\Gn}(\mu)$ is a closed subgroup of $\Gn$.
By Proposition~\ref{Serre.prop} the fixed point set $F:=(\An)^{\mu'}$ of every subgroup $\mu' \subset \mu$ is $\ZZ/2$-acyclic, in particular non-empty and  connected. We also know that $F$ is smooth and that $T_{a}F = (T_{a}\An)^{\mu'}$ since $\mu'$ is linearly reductive (see Proposition~\ref{Fogarty.prop}). If $a\in(\An)^{\mu}$, then the tangent representation of $\mu$ on $T_{a}\An$ is faithful, by Lemma~\ref{faithful.lem} above, and so $a$ is an isolated fixed point. Hence, $(\An)^{\mu}=\{a\}$. 

Choose generators $\sigma_{1},\ldots,\sigma_{n}$ of $\mu$ such that the images in $\GL(T_{a}\An)$ are reflections, i.e. have a single eigenvalue $-1$, and set $H_{i}:=(\An)^{\sigma_{i}}$. The tangent representation shows that $H_{i}$ is a hypersurface, hence defined by an irreducible polynomial $f_{i}\in\CC[x_{1},\ldots,x_{n}]$. Moreover, $\sigma_{i}^{*}(f_{i}) = -f_{i}$ and $\sigma_{i}^{*}(f_{j})=f_{j}$ for $j\neq i$. It follows that the linear subspace $V:=\CC f_{1}\oplus\cdots\oplus \CC f_{n}\subset \CC[x_{1},\ldots,x_{n}]$ is $\mu$-stable. In addition, any $\g\in G:=\C_{\Gn}(\mu)$ fixes $a$ and stabilizes all $\CC f_{i}$ and so, by the following Lemma~\ref{missing.lem} applied to the morphism $\phi:=(f_{1},\ldots,f_{n})\colon \An \to \An$, the action of $G$ on $\An$ is locally finite. Since $G$ is a closed subgroup of $\Gn$, it follows that it is an algebraic subgroup of $\Gn$, and its image in $\GL(V)$ is a closed subgroup contained in a maximal torus, hence a diagonalizable group.

Finally,  $\mm_{a} = V \oplus\mm_{a}^{2}$, and thus  the homomorphism $G \to \GL(T_a \AA^n)$ is injective, by Lemma~\ref{faithful.lem}. Hence the claim.
\end{proof}

\begin{rem} 
It is not difficult to show that the proposition holds for every finite commutative subgroup $\mu$ of rank $n$. In fact, the proof carries over to subgroups isomorphic to $\mu_{p}$ where $p$ is a prime, and every finite commutative subgroup $\mu$ of rank $n$ contains such a group.
\end{rem}

\begin{lem}\label{missing.lem}
Let $G\subset \Aut(\An)$ be a subgroup and let $\phi\colon \An \to X$ be a dominant morphism such that $\dim X = n$. Assume that $\phi^{*}(\OOO(X))$ is a $G$-stable subalgebra and that the induced action of $G$ on $X$ is locally finite. Then the same 
holds for the action of $G$ on $\An$.
\end{lem}

\begin{proof}
Put $A:=\phi^{*}(\OOO(X)) \subset  \CC[x_{1},\ldots,x_{n}]$ and denote 
by $R\subset  \CC[x_{1},\ldots,x_{n}]$ the integral closure of $A$. We first claim that 
the action of $G$ on $R$ is locally finite. In fact, let $f\in R$  and let $f^{m}+ a_{1}f^{m-1}+\cdots+a_{m}=0$ 
be an integral equation of $f$ over $A$. By assumption, the spaces $\langle G a_{i}\rangle$ are all finite 
dimensional, and so there is a $d\in\NN$ such that $\deg g a_{i} < d$ for all $g\in G$ and all $a_{i}$. 
Since $gf$ satisfies the equation $(gf)^{m}+ (ga_{1})(gf)^{m-1}+\cdots+(ga_{m})=0$ we 
get $\deg (gf) < d$ for all $g\in G$, hence the claim.

Therefore, we can assume that $X$ is normal and that $\phi\colon \An \to X$ is birational. Choose an open set $U\subset \An$ such that $\phi(U)\subset X$ is open and $\phi$ induces an isomorphism $U \simto \phi(U)$. Define $Y:=\bigcup_{g\in G} gU\subset \An$. Then the induced morphism  $\psi:=\phi|_{Y}\colon Y \to \phi(Y)$ is $G$-equivariant and a local isomorphism. This implies that  $\psi$ is a $G$-equivariant isomorphism. 

By assumption, the action of $G$ on $X$ is locally finite, and so $G$ is dense in an algebraic group $\bar G$ which acts regularly on $X$. Clearly, the open set $\phi(Y)$ is $\bar G$-stable and  thus the action of $\bar G$ on $\OOO(\phi(Y))$ is locally finite. Now the claim follows, because $\CC[x_{1},\ldots,x_{n}] \subset \OOO(Y)$ is a $G$-stable subalgebra.
\end{proof}
The proposition above has an interesting consequence for the linearization problem for finite group actions on affine 3-space $\AA^{3}$. In this case it is known that every faithful action of a non-finite reductive group on $\AA^{3}$ is linearizable (\name{Kraft-Russell}, see \cite{KrRu2012Families-of-Group-}).

\begin{cor} 
Let $\mu\subset \Gthree$ be a commutative subgroup of rank three. If the centralizer of  $\mu$ is not finite, then $\mu$ is conjugate to a subgroup of $D_3$.
\end{cor}

\par\medskip
%%%%%%%%%%%%%%%%%%%%%%%%%%%%%%%%%%%%%%%%%
\subsection{\texorpdfstring{$D_{n}$}{Dn}-stable unipotent subgroups}
Recall that every commutative unipotent group $U$ has a natural structure of a 
$\CC$-vector space, given by the exponential map $\exp\colon T_{e}U \simto U$. 
Thus $\Aut(U) = \GL(U)$ and  every action of an algebraic group on $U$ by group automorphisms 
is given by a linear representation. 

A (non-zero) locally nilpotent vector field $\delta=\sum_{i=1}^{n}h_{i}\frac{\partial}{\partial x_{i}}$
defines a (non-trivial) $\Cplus$-action on $\An$, hence a one-dimensional unipotent subgroup 
\[
U_{\delta} = \{ \Exp(t \delta) :=
(\exp( t \delta )(x_1), \ldots, \exp( t \delta )(x_n)) \mid t \in \Cplus \} \subseteq \Gn,
\]
and $U_{\delta}=U_{\delta'}$ if and only if $\delta'$ is a scalar multiple of $\delta$.
In the following we denote  by $e_1, \ldots, e_n$ the standard basis of $\ZZ^n$, and
 by  $\eps_1, \ldots, \eps_n$ the standard basis of the character group of $D_{n}$.

\begin{lem}\label{stableunipot.lem}
Let $U=U_{\delta} \subset \Gn$ be a one-dimensional unipotent subgroup. 
Then $U_{\delta}$ is normalized by $D_{n}$ if and only if  
$\delta$ is of the form $c x^{\gamma}\frac{\partial}{\partial x_{i}}$, 
where 
\[
x^{\gamma} = 
x_{1}^{\gamma_{1}}\cdots x_{i-1}^{\gamma_{i-1}}x_{i+1}^{\gamma_{i+1}}
\cdots x_{n}^{\gamma_{n}}
\]
and $c \in \Cst$. In particular, $U_{\delta}=\{(x_{1},\ldots,x_{i}+s(cx^{\gamma}),\ldots,x_{n})\mid s\in\CC\}$, and 
$\bd \circ \Exp(s \delta) \circ \bd^{-1}= \Exp(t^{e_i-\gamma}s \delta)$ for $\bd=(t_{1}x_{1},\ldots,t_{n}x_{n})\in D_{n}$.
\end{lem}

\begin{proof}
If $U_{\delta}$ is normalized by $D_{n}$, then $\bd^\ast \circ \delta \circ (\bd^{\ast})^{-1} \in \Cst \delta$ for all	
$\bd  \in D_n$. Writing 
$\delta = \sum_{i} h_i \frac{\partial}{\partial x_i}$ it follows that each $h_i$ is a monomial of the form
$h_i = a_i x^{\beta + e_i}$ for some $\beta\in\ZZ^{n}$. If $\beta_i \geq 0$ an induction on $m$ shows that,
for all $m\geq 1$, we have 
\[
\delta^{m}(x_i) = b^{(i)}_{m} x^{m \beta + e_i}, \text{ \ where }
b^{(i)}_{m} = a_i \prod_{l = 1}^{m-1} ( l b + a_i ) \text{ \ and \ } 
b = \sum_{j = 1}^{n} a_j \beta_{j}.
\]
Assume  that $\beta_i\geq 0$ for all $i$. Since $\delta$ is locally nilpotent there is a minimal $m_i \geq 0$ 
such that $b^{(i)}_{m_i+1} = 0$. This implies $a_i = -m_i b$. Since $\delta \neq 0$, we get 
\[		
0 \neq b = \sum_{i=1}^n a_i \beta_i = - b \sum_{i=1}^n  m_i \beta_{i},	
\]
and so $\sum m_i \beta_i= -1$. But this is a contradiction, because $m_i,\beta_{i} \geq 0$ for all $i$. 
Therefore $a_{i}x^{\beta+ e_i}\neq 0$ implies that $\beta_{j}\geq 0$ for all $j\neq i$, and $\beta_{i}=-1$. 
Thus there is only one term in the sum, i.e. 
$\delta=a_{i}x^\gamma\frac{\partial}{\partial x_{i}}$ where $\gamma:=\beta+e_i$ has the claimed form.
\end{proof}

\begin{rem}\label{characters.rem}
This lemma can also be expressed in the following way:
{\it There is a bijective correspondence between the $D_{n}$-stable 
one-dimensional unipotent subgroups $U\subset \Gn$ and the characters of $D_{n}$ of the form 
$\lambda=\sum_{j}\lambda_{j}\eps_{j}$ where one $\lambda_{i}$ equals $1$ 
and the others are $\leq 0$.} We will denote  this set of characters by $\Xu$:
\[
\Xu:=\{\lambda=\sum\lambda_{j}\eps_{j}\mid \exists\, i \text{ such that }\lambda_{i}=1
\text{ and }\lambda_{j}\leq 0\text{ for }j\neq i\}.
\]
If $\lambda\in\Xu$, then $U_{\lambda}$ denotes the corresponding one-dimensional unipotent subgroup normalized by $D_{n}$.
\end{rem}

\begin{rem} 
In \cite[Theorem 1]{Li2011Roots-of-the-affin} \name{Alvaro Liendo} shows that the 
locally nilpotent derivations normalized by the torus $D_{n}':= D_{n}\cap \SL_{n}$ have exactly the same form.
\end{rem}

\begin{lem}\label{translation.lem}
The subgroup $\Tn$ of translations is the only commutative unipotent subgroup normalized by $\GL_{n}$.
\end{lem}

\begin{proof} If $U \subset \Gn$ is a commutative unipotent subgroup normalized by $\GL_{n}$, then all the 
weights of the representation of $\GL_{n}$ on $T_{e}U\simeq U$ must belong to $X_{u}(D_{n})$.  
The dominant weights of $\GL_{n}$ are $\sum_{i}\lambda_{i}\eps_{i}$ where 
$\lambda_{1}\geq \lambda_{2}\geq \cdots\geq \lambda_{n}$, and only 
those of the form $\lambda=\eps_{1}+\sum_{i>1}\lambda_{i}\eps_{i}$ where 
$0\geq \lambda_{2}\geq\cdots\geq \lambda_{n}$  occur  in $\Xu$. If $\lambda\neq \eps_{1}$, 
i.e. $\lambda=\eps_{1}+ \lambda_{k}\eps_{k}+\lambda_{k+1}\eps_{k+1}+\cdots$ 
where $\lambda_{k}<0$, then the weight $\lambda':=(\lambda_{k}+1)\eps_{k}+\lambda_{k+1}\eps_{k+1}+\cdots$ 
is dominant and $\lambda'\prec\lambda$. Therefore  $\lambda'$ appears in the irreducible 
representation of $\GL_{n}$ of highest weight $\lambda$, but $\lambda'\notin\Xu$. Thus $U$ 
and $\Tn$ are isomorphic as $\GL_n$-modules, hence contain the same $D_n$-stable 
one-dimensional unipotent subgroups, and so $U=\Tn$.
\end{proof}

\par\medskip
%%%%%%%%%%%%%%%%%%%%%%%%%%%%%%%%%%%%%%%%%%%%%%%%%%
\subsection{Maximal tori}
It is clear that $D_{n} \subset \Gn$ is a maximal commutative subgroup of $\Gn$ since it coincides with its centralizer, see Lemma~\ref{lem1}. Moreover, \name{Bia{\l}ynicki-Birula} proved in \cite{Bi1966Remarks-on-the-act} that a faithful action of an $n$-dimensional torus on $\AA^n$ is linearizable (cf. \cite[Chap. I.2.4, Theorem 5]{KrSc1992Reductive-group-ac}).Thus we have the following result.

\begin{lem} \label{diagonalizable.lem}
$D_{n}$ is a maximal commutative subgroup of $\Gn$. 
Moreover, every algebraic subgroup of $\Gn$, which is isomorphic 
to $D_{n}$ is conjugate to $D_{n}$.
\end{lem}

Now let $G\subset \Gn$ be an algebraic subgroup which is normalized by $D_{n}$. Then the non-zero weights 
of the representation of $D_{n}$ on the Lie algebra $\Lie G$ belong to $\Xu$, and the weight spaces are 
one-dimensional. It follows that the non-zero weight spaces of $\Lie G$ are in bijective correspondence with 
the $D_{n}$-stable one-dimensional unipotent subgroups of $G$.

\begin{lem}\label{dimension.lem}
Let $G \subset \Gn$ be an algebraic subgroup normalized by a torus $D \subset \Gn$ of dimension $n$, let $U_{1},
\ldots,U_{r}$ be the $D$-stable one-dimensional unipotent subgroups of $G$, and put $X:= U_{1}\circ \cdots \circ U_{r}\subset G$.
\be
\item If $G$ is unipotent, then $G = X \circ X$ and $\dim G = r$.
\item If $D \subset G$, then $G^{0} = D\circ X \circ D \circ X$ and $\dim G = r + n$.
\ee
\end{lem}

\begin{proof} 
(a) The canonical map $U_{1}\times\cdots\times U_{r} \to G$ is dominant, and so $X \subset G$ 
is constructible and dense. Thus $X \circ X = G$, by Lemma~\ref{alggp.lem}, and $\dim G = \dim \Lie G = r$.

(b) Similarly, we see that $D\circ X \subset G^{0}$ is constructible and dense, and therefore  
$D\circ X \circ D \circ X = G^{0}$, and $\dim G = \dim \Lie G = \dim \Lie D + r$.
\end{proof}

\par\medskip
%%%%%%%%%%%%%%%%%%%%%%%%%%%%%%%%%%%%%%%%%%%%%
\subsection{Images of algebraic subgroups}
The next two propositions are crucial for the proof of our main theorem.

\begin{prop}\label{torusunipot.prop}
Let $\theta$ be an automorphism of $\Gn$. Then
\be
\item $D:=\theta(D_{n})$ is a torus of dimension $n$, conjugate to 
$D_{n}$.
\item If $U$ is a $D_n$-stable unipotent subgroup, then
$\theta(U)$ is a $D$-stable unipotent subgroup of the same dimension.
\item $\TTT:=\theta(\Tn)$ is a commutative unipotent subgroup of
dimension $n$, normalized by $D$, and the representation of 
$D$ on $\TTT$ is faithful.
\ee
\end{prop}

\begin{proof}
(a) We have $D_{n}=\C_{\Gn}(\mu_{2})$, by Lemma~\ref{lem1}, and thus 
$D=\theta(D_{n})=\C_{\Gn}(\theta(\mu_{2}))$. Proposition~\ref{crucial.prop} implies 
that $D$ is a diagonalizable algebraic subgroup with $\dim D \leq n$, hence  $D = D^0 \times F$ for some finite 
group $F$. If $k$ is prime to the order of $F$, then $\theta(\mu_{k}) \subset D^{0}$ and so
$\dim D^{0}= n$, because $\mu_{k}\simeq (\ZZ/k)^{n}$. Hence $D = D^{0}$ is an $n$-dimensional torus
which is conjugate to $D_{n}$, by Lemma~\ref{diagonalizable.lem}.

(b) First assume that  $\dim U=1$. Then $U$ consists of two $D_{n}$-orbits, 
$O:=U \setminus\{\id\}$ and $\{\id\}$. It follows that 
$\theta(U)$ consists of the two $D$-orbits $\theta(O)$ and $\{\id\}$, and so 
$\theta(U)$ is bounded constructible and thus a commutative
algebraic group normalized by $D$. Since it does not contain elements of 
finite order it is unipotent, and since it consists of only two $D$-orbits it is of 
dimension 1.

Now let $U$ be arbitrary,  $\dim U = r$, and let  $U_1, \ldots, U_r$ be the different   $D_{n}$-stable 
one-dimensional unipotent subgroups of $U$. Then  $X:=U_{1}\circ U_{2}\circ \cdots\circ U_{r} \subset  U$ is 
dense and constructible and $U = X \circ X$, by Lemma~\ref{dimension.lem}(a). Applying $\theta$ implies
 that $\theta(X) = \theta(U_{1})\circ\cdots\circ\theta(U_{r})$ is bounded constructible and connected, as well 
 as $\theta(U)=\theta(X)\circ\theta(X)$, and thus $\theta(U)$ is a connected algebraic subgroup of $\Gn$ 
 normalized by $D$. Since every element of $\theta(U)$ has infinite order, $\theta(U)$ must be unipotent. 
 Moreover, $\dim \theta(U)\geq r$, since $\theta(U)$ contains the $D$-stable one-dimensional unipotent 
 subgroups $\theta(U_{i})$, $i=1,\ldots,r$. The same argument applied to $\theta^{-1}$ finally gives $\dim \theta(U) = r$.
 
(c)  This statement follows from (b) and the fact that $\Tn$ contains a dense $D_{n}$-orbit with trivial stabilizer.
\end{proof}

The same arguments, this time using Lemma~\ref{dimension.lem}(b), gives the next result.

\begin{prop}\label{algsubgroups.prop}
Let  $\theta$ be an automorphism of $\Gn$ and let 
$G \subset \Gn$ be an algebraic subgroup 
which contains a torus $D$ of dimension $n$.
\be
\item The image $\theta(G)$ is an algebraic subgroup of $\Gn$ of the same dimension
$\dim G$. 
\item We have  $\theta(G^{0})=\theta(G)^{0}$. In particular, $\theta(G)$ is connected if $G$  is connected.
\item  If $G$ is reductive, then so is $\theta(G)$, and then $\theta(G)$ is conjugate to a closed 
subgroup of $\GL_{n}$.
\ee
\end{prop}

\begin{proof}
As above, let  $U_1, \ldots, U_r$ be the different   $D$-stable 
one-dimensional unipotent subgroups of $G$, and put $X:=U_{1}\circ\cdots\circ U_{r}$. Then
$D \circ X$ is constructible in $G^{0}$, and $D \circ X\circ D \circ X=G^{0}$, by Lemma~\ref{dimension.lem}(b).
Applying $\theta$ we see that $\theta(D \circ X\circ D \circ X) =\theta(D)\circ\theta(X)\circ\theta(D)\circ\theta(X)$ 
is bounded constructible and connected, and so $\theta(G^{0})$ is a connected algebraic subgroup of $\Gn$, 
of finite index in $\theta(G)$. Since the $\theta(U_{i})$ are different $\theta(D)$-stable one-dimensional unipotent 
subgroups of $\theta(G)$ we have $\dim\theta(G) \geq \dim\theta(D) + r=\dim G$. Using $\theta^{-1}$ we get equality. 
This proves (a) and (b).

For (c) we remark that if $G$ contains a normal unipotent subgroup $U$, then $\theta(U)$ is a normal unipotent 
subgroup of $\theta(G)$. Moreover, a reductive subgroup $G$ containing a torus of dimension $n$ has no 
non-constant invariants, and so  $G$ is linearizable (see \cite[Proposition~5.1]{KrPo1985Semisimple-group-a}).
\end{proof}

\par\medskip
%%%%%%%%%%%%%%%%%%%%%%%%%%%%%%%%%%%%%%%%%
\subsection{Proof of the Main Theorem}\label{proof.subsec}
Let $\theta$ be an automorphism of $\Gn$. It follows from Proposition~\ref{algsubgroups.prop} 
that there is a $\g\in\Gn$ such that $\g\circ\theta(\GL_{n})\circ\g^{-1}\subset \GL_{n}$. Therefore 
we can assume that $\theta(\GL_{n})=\GL_{n}$. The subgroup $\Tn$ of translations is the only 
commutative unipotent subgroup normalized by $\GL_{n}$, by Lemma~\ref{translation.lem}. 
Therefore, $\theta(\Tn)=\Tn$ and so $\theta(\Aff_{n})=\Aff_{n}$. Now the theorem follows from the next proposition.
\qed

\begin{prop}\label{autaff.prop}
\be 
\item Every automorphism $\theta$ of $\Aff_{n}$ 
with	$\theta(\GL_{n}) = \GL_{n}$ and $\theta(\Tn) = \Tn$ is of the 
form $\theta (\f) = \tau(\g\circ \f \circ \g^{-1})$ where $\g \in \GL_{n}$ and 
$\tau$ is an automorphism of the field $\CC$.
\item If $\theta$ is an automorphism of $\Gn$ such 
that $\theta|_{\Aff_{n}}=\Id_{\Aff_{n}}$, then $\theta|_{\Jn}=\Id_{\Jn}$.
\ee
\end{prop}

\begin{proof}
(a)  It is enough to prove that $\theta(\f) = \g \circ \tau(\f) \circ \g^{-1}$
for some $\g \in \GL_n$ and some automorphism $\tau\colon \CC \to \CC$ 
of the field $\CC$.
Let $Z=\Cst \subseteq \GL_{n}$ be the center of $\GL_n$ 
and define $\theta_0 := \theta |_{Z}\colon Z \to Z$, 
$\theta_1 := \theta |_{\Tn}\colon \Tn \to \Tn$.
It follows that $\theta_0$ and $\theta_1$ are abstract group homomorphisms
of $\Cst$ and $\Tn$ respectively, and for all $c \in \Cst$,
$\bt \in \Tn$ we get
\[
\label{equat}\tag{$\ast$}
\theta_1( c \cdot \bt) = \theta_1(c \circ \bt \circ c^{-1}) = 
\theta_0(c) \circ \theta_1(\bt) \circ \theta_0(c)^{-1} =
\theta_0(c) \cdot \theta_1(\bt) \, ,
\]
where ``$\,\cdot \,$'' denotes scalar multiplication. We claim that
$\tau\colon \CC \to \CC$ defined by $\tau |_{\Cst} = \theta_0$, $\tau(0) = 0$, 
is an automorphism of the field $\CC$. Indeed, using \eqref{equat} one sees that
$\tau(a + b) = \tau(a) + \tau(b)$ for all $a, b \in \CC^\ast$ such that $a+b \neq 0$.
As $\theta_0(-1) = -1$ it follows that $\tau(-a) = -\tau(a)$ for all $a \in \CC^\ast$
and so $\tau(a + (- a)) = \tau(a) + \tau(-a)$. This implies the claim.

Thus we can assume that $\theta_0 = \id_{\Cst}$. Using \eqref{equat} again, it
follows that $\theta_1$ is linear.  Considering $\theta_1$ as an element
of $\GL_n$ we have 
$\theta_1( \bt ) = \theta_1 \circ \bt \circ \theta_1^{-1}$, and thus
we can assume that $\theta_1 = \id_{\Tn}$. But this implies that 
$\theta(\g) = \g$ for all $\g \in \GL_n$, because 
\[
\g\circ \bt \circ \g^{-1}= \theta(\g\circ \bt\circ \g^{-1})  
=  \theta(\g) \circ \bt \circ \theta(\g)^{-1} 
\]
for all $\bt\in\Tn$.
	
(b) Let $U\subset \Gn$ be a one-dimensional unipotent $D_n$-stable 
subgroup. We first claim that $\theta(U) = U$ and that $\theta|_{U}$ is linear. 
In fact, $U':=\theta(U)$ is a one-dimensional unipotent $D_{n}$-stable subgroup,
by Proposition~\ref{torusunipot.prop}(b), and the characters $\lambda$ and 
$\lambda'$ associated to $U$ and $U'$ (see Remark~\ref{characters.rem}) have the same kernel, because
\[
\label{eq}
\tag{$\ast\ast$}
\theta(\lambda(\bd) \cdot u)
= \theta(\bd \circ u \circ \bd^{-1})
= \bd \circ \theta(u) \circ \bd^{-1}
= \lambda'(\bd) \cdot \theta(u)
\text{ \ for $\bd \in D_n$, $u \in U$}.
\]
Hence $\lambda = \pm \lambda'$. If $\lambda = -\lambda'$,
then  $U \subseteq \GL_n$ and so $U' = U$, since
$\theta|_{\GL_{n}}=\Id_{\GL_{n}}$, hence a contradiction. 
Thus $\lambda = \lambda'$, and so $U=U'$ and (\ref{eq}) shows that 
$\theta |_U$ is linear, proving our claim.
	
As a consequence, 
$\theta|_{U_{\lambda}}= a_{\lambda}\Id_{U_{\lambda}}$ for all $\lambda\in\Xu$, with suitable $a_{\lambda}\in\Cst$.
If $\lambda_{i}=1$ put $\gamma_{i}:=0$ and $\gamma_{j}:= -\lambda_{j}$. Then $\f=(x_1, \ldots, x_i + x^\gamma, \ldots ,x_n)
\in U_{\lambda}$, see Lemma~\ref{stableunipot.lem}.
Conjugation with the translation $\bt\colon x \mapsto x - \sum_{j \neq i} e_j$  gives
\[
\bt \circ \f \circ \bt^{-1} = (x_1, \ldots , x_i + h_\gamma , \ldots, x_n) 
\text{ where }
h_\gamma := (x_1+1)^{\gamma_1} (x_2+1)^{\gamma_2} \cdots (x_n+1)^{\gamma_n}.
\]
Now we apply $\theta$ to get 
$\theta(\bt \circ \f \circ \bt^{-1}) = \bt \circ \theta(\f) \circ \bt^{-1}$. 
Since all the monomials $x^{\gamma'}$ with $\gamma' \leq \gamma$ appear in 
$h_\gamma$ it follows that the corresponding coefficients $a_{\lambda'}$ must all be equal.
In particular, $a_\lambda = a_{\eps_{i}} = 1$ since $U_{\eps_{i}}\subset \Tn$. 
This shows  that $\theta |_{\Jn} = \Id_{\Jn}$.
\end{proof}

\par\bigskip
%%%%%%%%%%%%%%%%%%%%%%%%%%%%%%%%%%%%%%%%%%%%%%%%%%%%%%

\vskip1cm
%\bibliography{/Users/hkraft/Documents/Literatur/HP-Bib}
%\bibliography{KraftStampfli}

\begin{thebibliography}{BCW82}

\bibitem[BCW82]{BaCoWr1982The-Jacobian-conje}
Hyman Bass, Edwin~H. Connell, and David Wright, \emph{The {J}acobian
  conjecture: reduction of degree and formal expansion of the inverse}, Bull.
  Amer. Math. Soc. (N.S.) \textbf{7} (1982), no.~2, 287--330.
  
\bibitem[BKY13]{BeYu2013On-The-Zariski-Top}
Aalexei~Belov-Kanel and Jie-Tai Yu, \emph{On The Zariski Topology Of 
Automorphism Groups Of Affine Spaces And Algebras}, 2013,
  http://arxiv.org/abs/1207.2045.
  
\bibitem[BB66]{Bi1966Remarks-on-the-act}
Andrzej~Bia{\l}ynicki-Birula, \emph{Remarks on the action of an algebraic torus on
 {$k^{n}$}}, Bull. Acad. Polon. Sci. S{\'e}r. Sci. Math. Astronom. Phys.
 \textbf{14} (1966), 177--181.

\bibitem[D{\'e}s06]{De2006Sur-le-groupe-des-}
Julie D{\'e}serti, \emph{Sur le groupe des automorphismes polynomiaux du plan
  affine}, J. Algebra \textbf{297} (2006), no.~2, 584--599.
  
  \bibitem[Fog73]{Fo1973Fixed-point-scheme}
John Fogarty, \emph{Fixed point schemes}, Amer. J. Math. \textbf{95} (1973),
  35--51.
  
\bibitem[KP85]{KrPo1985Semisimple-group-a}
Hanspeter Kraft and Vladimir~L. Popov, \emph{Semisimple group actions on the
 three-dimensional affine space are linear}, Comment. Math. Helv. \textbf{60}
 (1985), no.~3, 466--479.
 
 \bibitem[KR12]{KrRu2012Families-of-Group-}
Hanspeter Kraft and Peter Russell, \emph{Families of Group Actions, Generic Isotriviality, and Linearization}, Submitted 2012.

 \bibitem[KS92]{KrSc1992Reductive-group-ac}
Hanspeter Kraft and Gerald~W. Schwarz, \emph{Reductive group actions with
  one-dimensional quotient}, Inst. Hautes {\'E}tudes Sci. Publ. Math. (1992),
  no.~76, 1--97.

 \bibitem[Kum02]{Ku2002Kac-Moody-groups-t}
Shrawan Kumar, \emph{Kac-{M}oody groups, their flag varieties and
  representation theory}, Progress in Mathematics, vol. 204, Birkh{\"a}user
  Boston Inc., Boston, MA, 2002.
  
  \bibitem[Lie11]{Li2011Roots-of-the-affin}
Alvaro Liendo, \emph{Roots of the affine cremona group}, Transform. Groups, to
  appear (2011).

\bibitem[Ser09]{Se2009How-to-use-finite-}
Jean-Pierre Serre, \emph{How to use finite fields for problems concerning
  infinite fields}, Arithmetic, geometry, cryptography and coding theory,
  Contemp. Math., vol. 487, Amer. Math. Soc., Providence, RI, 2009,
  pp.~183--193.
  
  \bibitem[Smi34]{Sm1934A-theorem-on-fixed}
Paul~A. Smith, \emph{A theorem on fixed points for periodic transformations},
  Ann. of Math. (2) \textbf{35} (1934), no.~3, 572--578.


  

\end{thebibliography}
%\bibliographystyle{amsalpha}

\end{document}